\documentclass[reqno]{amsart}
\usepackage{amsmath}
\usepackage{amssymb, amsthm, amsfonts, tikz-cd, mathrsfs,mathtools,stmaryrd,enumitem, algorithmic,float,comment,tikz,hyperref}
\usepackage[toc]{appendix}
\usetikzlibrary{backgrounds}
\usetikzlibrary{decorations.pathreplacing}
\usepackage{fullpage}
\usepackage{xcolor}

\usepackage{tikz}
\usetikzlibrary{calc,decorations.pathreplacing}

\hypersetup{
  colorlinks   = true, 
  urlcolor     = blue, 
  linkcolor    = blue, 
  citecolor   = red 
}

\usepackage{comment}

\usepackage[capitalize,nameinlink]{cleveref}
\crefname{theorem}{Theorem}{Theorems}
\crefname{lemma}{Lemma}{Lemmas}
\crefname{claim}{Claim}{Claims}
\crefname{prop}{Proposition}{Propositions}
\crefname{figure}{Figure}{Figures}

\newtheorem{theorem}{Theorem}

\newtheorem{lemma}[theorem]{Lemma}

\newtheorem{problem}[theorem]{Problem}

\newtheorem*{claim*}{Claim}

\theoremstyle{remark}
\newtheorem*{remark*}{Remark}

\theoremstyle{definition}
\newtheorem{definition}[theorem]{Definition}

\numberwithin{theorem}{section}

\renewcommand{\phi}{\varphi}

\newcommand{\cM}{\mathcal M}

\def\1{\mathbbm{1}}

\renewcommand{\le}{\leqslant}
\renewcommand{\ge}{\geqslant}

\newcommand{\dist}{\operatorname{dist}}

\newcommand{\lpr}[1]{\left(#1\right)}

\newcommand{\diam}{\operatorname{diam}}
\newcommand{\Tr}{\operatorname{tr}}
\newcommand{\arcosh}{\operatorname{arcosh}}

\newcommand{\ov}[1]{\overline{#1}}

\usepackage[
backend=biber,
style=numeric,
maxnames=99,
giveninits=true
]{biblatex}
\renewbibmacro{in:}{%
  \ifentrytype{article}{}{\printtext{\bibstring{in}\intitlepunct}}}

\title
{Generalized Short Monochromatic Odd Cycles} 

\author{Sammy Luo}
\thanks{Binghamton University. Email: \texttt{sammy.luo@binghamton.edu}}
\author{Zixuan Xu}
\thanks{Princeton University. Email: \texttt{zx6683@princeton.edu}}

\begin{document}

\maketitle
\vspace{-1.2em}

\begin{abstract}
    Answering a question of Noga Alon, we generalize a result of Janzer and Yip to show that every $k$-edge-coloring of $K_{r^k+1}$ contains a monochromatic subgraph with chromatic number at least $r+1$ and diameter $O(k^{3/2}r^{k/2}\log r)$.
\end{abstract}

\section{Introduction}
The inspiration for our work in this paper stems from the multicolor Ramsey number question for odd cycles: For a fixed $k$, what is the smallest $n$ such that every $k$-coloring of the edges of $K_n$ contains a monochromatic odd cycle? A simple observation shows that $K_{2^k}$ can be $k$-edge-colored so that every color class is bipartite; for example, one can label the vertices with length $k$ binary strings, then color each edge $vw$ based on the leftmost bit at which the labels of $v$ and $w$ differ. This construction is quite tight, in the sense that moving any edge to a new color class will form many odd cycles, and in fact many triangles, in that color.

Meanwhile, every $k$-edge-coloring of $K_{2^k+1}$ contains a monochromatic odd cycle, but the simplest way to see this is via a very non-local argument using the submultiplicativity of the chromatic number when taking unions of graphs, which gives no information on the size of the odd cycle found. 
Motivated by these observations, Erd\H{o}s and Graham asked the following question.

\begin{problem}[Erd\H{o}s--Graham~\cite{ErdosGraham1975PartitionTheorems}]\label{problem:odd-cycle}
    What is the smallest $L(k)$ such that every $k$-edge-coloring of the complete graph $K_{2^k+1}$ has a monochromatic odd cycle of length at most $L(k)$?
\end{problem}

The first nontrivial upper bound on $L(k)$ was proven by Gir\~ao and Hunter \cite{girão2024monochromaticoddcyclesedgecoloured}, who showed that $L(k)\le \frac{2^k+1}{k^{1-o(1)}}$. This was then improved by Janzer and Yip \cite{JanzerYip2026ShortMonochromaticOddCycles} to the current best known upper bound of $L(k)=O(k^{3/2}2^{k/2})$. Meanwhile, the best known lower bound is given by Day and Johnson \cite{DayJohnson2017MulticolourRamseyOddCycles}, who showed that $L(k)\ge 2^{\Omega(\sqrt{\log k})}$.

Noga Alon~\cite{Noga1} asked whether these results could be generalized by asking an analogous question about monochromatic subgraphs with chromatic number at least $r+1$ in a $k$-edge-coloring of $K_{r^k+1}$: How ``small'' of a monochromatic subgraph $H$ with chromatic number $\ge r+1$ can we find in such a coloring? While there are multiple notions of ``small'' that seem like potential natural generalizations of the length of the shortest monochromatic odd cycle to higher chromatic numbers, the one that turns out to be useful to us looks at the \emph{diameter} $\diam(H)$ of $H$, which is defined as the maximum length of the shortest path between some two vertices $u,v\in V(H)$. The specific question we study is the following generalization of \cref{problem:odd-cycle}.

\begin{problem}
    What is the smallest $L_r(k)$ such that every $k$-edge-coloring of the complete graph $K_{r^k+1}$ has a monochromatic subgraph $H$ of diameter at most $L_r(k)$ with chromatic number $\chi(H) \ge r+1$?
\end{problem}
Note that the case $r = 2$ directly corresponds to \cref{problem:odd-cycle}, since a graph $G$ contains a subgraph with chromatic number at least $3$ and diameter $\le d$ if and only if it contains an odd cycle of length $\le 2d+1$. 

The reason that diameter of $H$ turns out to be a more natural quantity to study than, say, the number of vertices in $H$, comes from our need to generalize the following key property of small odd cycles which is used in \cite{JanzerYip2026ShortMonochromaticOddCycles}: A graph $G$ with no small odd cycles looks locally bipartite in a way that is captured \emph{spectrally}, by the fact that small odd powers of the adjacency matrix $A_G$ (or in fact of any matrix with zeros in the same places as $A_G$) have all zeros along their diagonal, due to $G$ not having any short closed walks of odd length. The corresponding property for larger $r$ is captured in \cref{lem:poly-local} and \cref{lem:trace} below.

Our main result is the following upper bound on $L_r(k)$.

\begin{theorem}\label{thm:ub}
 For all integers $r\ge 2$ and $k\ge 1$, we have
 \[L_r(k)= O\lpr{k^{3/2}r^{k/2}\log r}.\]
 In particular, for fixed $r$, we have $L_r(k)= O_r\lpr{k^{3/2}r^{k/2}}$.
\end{theorem}

We obtain this bound by building on the framework of Janzer and Yip's argument in~\cite{JanzerYip2026ShortMonochromaticOddCycles}. As in their argument, our result follows from a more general bound that also applies to complete graphs with a larger number of vertices.

\begin{theorem}\label{thm:ub-gen}
 Let $r\ge 2$ and $k\ge 1$ be integers, and let $0<\delta \le 1$ such that $n=(1+\delta)r^k$ is an integer. Then in every $k$-edge-coloring of $K_n$ there is a monochromatic subgraph with chromatic number at least $r+1$ and diameter $O(k^{3/2}\delta^{-1/2}\log r)$.
\end{theorem}
When $r=2$, we recover a bound of $O(k^{3/2}\delta^{-1/2})$ on the length of the shortest monochromatic odd cycle in every $k$-edge-coloring of $K_n$, which matches the bound of $4k^{3/2}\delta^{-1/2}$ in \cite[Theorem~1.5]{JanzerYip2026ShortMonochromaticOddCycles} up to a constant.

\cref{thm:ub} follows by taking $\delta=r^{-k}$ in \cref{thm:ub-gen}.

In addition, by adapting the argument of Day and Johnson~\cite{DayJohnson2017MulticolourRamseyOddCycles}, we show the following lower bound.

\begin{theorem}\label{thm:lb}
    For all integers $r\ge 2$ and $k\ge 1$, we have
    \[L_r(k)\ge 2^{\Omega(\sqrt{\log k})}.\]
\end{theorem}

The rest of the paper proceeds as follows: We begin by introducing some tools we need, most notably the Lov\'asz theta function and the Chebyshev polynomials, along with some of their useful properties, in~\cref{sec:prelims}. We then give the proof of \cref{thm:ub-gen} in \cref{sec:ub} and the proof of \cref{thm:lb} in \cref{sec:lb}. We end with some concluding remarks in \cref{sec:conclusion}.

\section{Preliminaries}
\label{sec:prelims}

Let $G$ be a graph. We use $V(G)$ and $E(G)$ to denote the vertex set and edge set of $G$, respectively. For $v\in V(G)$, we use
\[B_G(v,\rho)=\{u\in V(G):\dist_G(u,v)\le \rho\}\]
to denote the ball of radius $\rho$ centered at $v$, where $\dist_G(u,v)$ is the number of edges in the shortest path between $u$ and $v$. We view $B_G(v,\rho)$ as the induced subgraph on this vertex set. We use $\chi(G)$ to denote the chromatic number of $G$.

\subsection{Lov\'asz theta function} 

\begin{definition}[Orthonormal representation]
    For a graph $G$, an \emph{orthonormal representation} $X$ is a set of unit vectors $\{x_v\}_{v\in V(G)}$ in a Euclidean space such that $x_u\cdot x_v = 0$ whenever $u\ne v\in V(G)$ and $uv\not\in E(G)$.
\end{definition}

For a graph $G$, denote by $\mathcal M(G)$ 
the set of Gram matrices $M[X]\in \mathbb{R}^{V(G)\times V(G)}$ given by $M[X]_{u,v} = x_u\cdot x_v$ for some orthonormal representation $X$ of $G$. In particular, every $M\in \mathcal{M}(G)$ is symmetric and positive semidefinite (i.e. all its eigenvalues are nonnegative) and satisfies $M_{vv}=1$ for all $v$ and $M_{uv}=0$ for all $u\neq v$ such that $uv\notin E(G)$.

\begin{definition}[Lov\'asz theta function]
    For a graph $G$, the Lov\'asz theta function $\vartheta(G)$ is defined as 
    \[\vartheta(G) := \min_{c,X}\max_{v\in V(G)}(c\cdot x_v)^{-2}\]
    where the minimum is over all unit vectors $c\in \mathbb{R}^{V(G)}$ and orthonormal representations $X$ of $G$.
\end{definition}

 We recall several useful properties of the Lov\'asz theta function $\vartheta(G)$ of a graph $G$, the proofs of which can be found in \cite{JanzerYip2026ShortMonochromaticOddCycles}.

\begin{lemma}\label{lem:spectral-theta}
    For a graph $G$, we have $\vartheta(\ov G)=\max_{M\in\mathcal M(G)}\lambda_1(M)$, where $\ov G$ denotes the complement of $G$, and $\lambda_1(M)$ is the largest eigenvalue of $M$.
\end{lemma}

\begin{lemma}\label{lem:theta-complete-graph}
    For a complete graph $K_n$ on $n$ vertices, we have $\vartheta(\ov K_n) = n$.
\end{lemma}

\begin{lemma}\label{lem:submultiplicative}
    For graphs $G_1, G_2$ on the vertex set $[n]$, let $G_1\cup G_2$ denote the graph on vertex set $[n]$ with edge set $E(G_1)\cup E(G_2)$. Then we have $\vartheta(\ov{G_1\cup G_2})\le \vartheta(\ov{G_1})\cdot \vartheta(\ov{G_2})$. Consequently, if $G_1,\dots,G_k$ are graphs on the same vertex set $[n]$, and  $E(G_1)\cup\cdots\cup E(G_k) = E(K_n)$, then $\prod_{i=1}^k \vartheta(\ov{G_i})\ge n$.
\end{lemma}

\subsection{Chebyshev polynomial}
The degree-$\rho$ Chebyshev polynomial $T_\rho$ is defined to satisfy
\[T_\rho(\cos \theta)=\cos(\rho\cdot\theta).\]
Concretely, for $x\ge 1$, we have
\[T_\rho(x)=\cosh\lpr{\rho\cdot \arcosh x}=\frac{1}{2}\left(\left(x+\sqrt{x^2-1}\right)^\rho+\left(x-\sqrt{x^2-1}\right)^\rho\right).\]
The key property we use is that $|T_\rho(x)|\le 1$ for all $x\in [-1,1]$.

\section{Proof of \cref{thm:ub}}
\label{sec:ub}

We begin with a few lemmas. The first relates the chromatic number of an induced subgraph of $G$ to the largest eigenvalue of a corresponding Gram matrix. When we write $M[H]$, or ``$M$ restricted to $H$'', we mean the submatrix of $M$ consisting of the rows and columns corresponding to the vertex set of $H$.

\begin{lemma}\label{lem:local-r}
Let $H$ be an induced subgraph of $G$. If $\chi(H)\le r$, then for each $M\in\mathcal M(G)$, we have $\lambda_1(M[H])\le r$.
\end{lemma}

\begin{proof}
    \cref{lem:spectral-theta} yields $\lambda_1(M[H])\le \vartheta(\ov H)$, so this result follows immediately from the well-known fact that $\vartheta(\ov H)\le \chi(H)$.
\end{proof}

Let $e_v$ denote the standard unit vector whose $v$-coordinate is $1$, with all other coordinates equal to $0$. The following lemma shows that to evaluate $p(M)e_v$ for a polynomial $p$ of degree $\le \rho$, it suffices to evaluate $p$ on the submatrix of $M$ restricted to the ball of radius $\rho$ around $v$.

\begin{lemma}\label{lem:poly-local}
For a graph $G$, let $M\in\mathcal M(G)$ and let $\rho\ge 0$ be an integer. For $v\in V(G)$, let $B_v=B_G(v,\rho)$.
If $p$ is a polynomial of degree at most $\rho$, then we have
\[p(M)e_v=p(M[B_v])e_v,\]
where the right-hand side is extended by zero outside $B_v$.
\end{lemma}

\begin{proof}
Since $M_{uw}=0$ unless $u=w$ or $uw\in E(G)$, for each positive integer $s$, the vector $M^s e_v$ is supported on the coordinates corresponding to vertices at graph distance at most $s$ from $v$. Thus, $(M^s)_{wv}$ is a weighted sum over length-$s$ walks from $v$ to $w$, and therefore $(M^s)_{wv}=0$ if $\dist_G(v,w)>s$.

If $s\le \rho$, then every walk of length $s$ starting at $v$ remains inside $B_G(v,\rho)$.  Hence we have $M^s e_v=M[B_v]^s e_v$ for every $0\le s\le \rho$. The result follows by linearity for every polynomial of degree at most $\rho$.
\end{proof}

Now we bound the trace of $p(M)$ for a certain type of polynomial $p$ of degree at most $\rho$.

\begin{lemma}\label{lem:trace}
Let $G$ be a graph on $n$ vertices and let $M\in \cM(G)$. Let $p$ be a polynomial of degree at most $\rho$ such that $|p(x)|\le 1$ for all $x\in [0,r]$. Suppose the ball of radius $\rho$ around $v$ for every $v\in V(G)$ is $r$-colorable in $G$. Then we have
\[ \Tr\bigl(p(M)^2\bigr)\le n.\]
\end{lemma}

\begin{proof}
Fix $v\in V(G)$ and let $B_v=B_G(v,\rho)$ for convenience. By assumption, $B_v$ is $r$-colorable. Thus by \cref{lem:local-r}, we have $\lambda_1(M[B_v])\le r$.
Furthermore, since $M$ is positive semidefinite, so is $M[B_v]$, so every eigenvalue of $M[B_v]$ lies in $[0,r]$.  Since $|p(x)|\le 1$ for all $x\in [0,r]$, we claim that \[\|p(M[B_v])\|_{\mathrm{op}}\le 1.\] 
Indeed, let $\lambda_1,\dots, \lambda_m\in [0,r]$ be the eigenvalues of $M[B_v]$. Then $p(\lambda_1),\dots, p(\lambda_m)\in [-1,1]$ are the eigenvalues of $p(M[B_v])$. So we have $\|p(M[B_v])\|_{\mathrm{op}} = \max_{i\in [m]} |p(\lambda_i)|\le 1$.
By \cref{lem:poly-local}, we have $p(M)e_v=p(M[B_v])e_v$, and thus $\|p(M)e_v\|_2\le 1$.
Summing over all the vertices, we obtain
\[\Tr(p(M)^2)=\sum_{v\in V(G)}e_v^Tp(M)^2e_v=\sum_{v\in V(G)}\|p(M)e_v\|_2^2\le n.\qedhere\]
\end{proof}

Now we are ready to give an upper bound for $\vartheta(\ov G)$ when every ball of radius $\rho$ in $G$ is $r$-colorable.

\begin{lemma}\label{lem:cheb}
Let $G$ be a graph on $n$ vertices. Suppose every ball of radius $\rho$ in $G$ is $r$-colorable. Then we have
\[
\vartheta(\ov G)\le r + \frac{r}{4}((2\sqrt{n})^{1/\rho}-1)^2.
\]
\end{lemma}

\begin{proof}
Choose
\[p(x)=T_\rho\left(\frac{2x}{r}-1\right).\]
Then $p$ has degree $\rho$. Furthermore, for $x\in[0,r]$, we have $|p(x)|\le 1$, since $\frac{2x}{r}-1\in[-1,1]$.
By \cref{lem:trace}, for every $M\in\mathcal M(G)$ we have
\[\Tr\bigl(p(M)^2\bigr)\le n.\]

By \cref{lem:spectral-theta}, we can choose $M\in\mathcal M(G)$ such that $\lambda_1(M)=\vartheta(\ov G)$. For convenience, let $\lambda=\lambda_1(M)$. If $\lambda\le r$, there is nothing to prove.  Otherwise, suppose $\lambda=r+\eta$ for $\eta > 0$. Then $p(M)$ has an eigenvalue
\[p(r+\eta)=T_\rho\left(1+\frac{2\eta}{r}\right).\]
Hence we must have
\[
T_\rho\left(1+\frac{2\eta}{r}\right)^2\le\Tr\bigl(p(M)^2\bigr)\le n.\]
Thus $T_\rho\left(1+\frac{2\eta}{r}\right)\le\sqrt{n}$.
For $x\ge 1$, we have 
\[T_\rho(x)=\frac{1}{2}\left(\left(x+\sqrt{x^2-1}\right)^\rho+\left(x-\sqrt{x^2-1}\right)^\rho\right)\ge \frac{1}{2}\left(x+\sqrt{x^2-1}\right)^\rho.\]
Substituting $x=1+\frac{2\eta}{r}$ gives
\[1+2\sqrt{\frac{\eta}{r}} \le 1+\frac{2\eta}{r}+2\sqrt{(\frac{\eta}{r})^2+\frac{\eta}{r}}\le \left(2T_\rho\left(1+\frac{2\eta}{r}\right)\right)^{1/\rho} \le(2\sqrt{n})^{1/\rho}.\]
Further rearranging yields
\[
\eta \le \frac{r}{4}((2\sqrt{n})^{1/\rho}-1)^2,
\]
as desired.
\end{proof}

We are now ready to prove \cref{thm:ub-gen}.

\begin{proof}[Proof of \cref{thm:ub-gen}]
Let $N=(1+\delta)r^k$ for convenience. Fix a $k$-edge-coloring of $K_N$ with color classes $G_i$ for $i\in \{1,\dots,k\}$, so $K_N=G_1\cup\cdots\cup G_k$. By \cref{lem:submultiplicative}, we have $\prod_{i=1}^k\vartheta(\ov{G_i})\ge N$. Then for some color $j$, we have
\[
\vartheta(\ov G_j)\ge N^{1/k}.
\]
Without loss of generality, assume $j=1$. Let $\rho\ge 0$ be the largest integer such that every ball of radius $\rho$ in $G_1$ is $r$-colorable. Note that $\rho$ is finite because otherwise, in the limit as $\rho\to \infty$, \cref{lem:cheb} yields $\vartheta(\ov G)\le r<N^{1/k}$, a contradiction. If $\rho=0$, then $G_1$ contains a subgraph of diameter at most $2$ that is not $r$-colorable, and we are done, so assume $\rho>0$. By \cref{lem:cheb}, we have
\[
r+\frac{r}{4}((2\sqrt{N})^{1/\rho}-1)^2 \ge N^{1/k}=r(1+\delta)^{1/k}.
\]
Since $\delta \le 1$, we have $r(1+\delta)^{1/k}\ge r \exp(\delta/(2k))$. Then we have
\[
\exp(\delta/(2k))\le 1+\frac{1}{4}((2\sqrt{N})^{1/\rho}-1)^2\le \exp(\frac{1}{4}((2\sqrt{N})^{1/\rho}-1)^2).
\]
Rearranging then yields
\[
1+\sqrt{\frac{2\delta}{k}}\le (2\sqrt{N})^{1/\rho},
\]
so
\[
\rho \le \frac{\log (2\sqrt{N})}{\log (1+\sqrt{\frac{2\delta}{k}})}\le \frac{\frac{1}{2} k\log r+O(1)}{\frac{1}{2}\sqrt{\frac{2\delta}{k}}}=O(k^{3/2}\delta^{-1/2}\log r).
\]
Then by our definition of $\rho$, taking a ball of radius $\rho+1$ yields a subgraph $H$ of $G_1$ with $\chi(H)\ge r+1$ and  $\diam(H)\le 2\rho+2$.
\end{proof}

\section{Lower bound}
\label{sec:lb}

Our lower bound construction is a direct generalization of Day and Johnson's construction for the case $r=2$ in
\cite{DayJohnson2017MulticolourRamseyOddCycles}. See also \cite{zhu2025}, which contains an equivalent algebraic reframing of Day and Johnson's construction.

As in \cite{DayJohnson2017MulticolourRamseyOddCycles}, we build our construction inductively. The role that the odd cycle takes in Day and Johnson's construction is taken here by what we call a \emph{bracelet graph} (named for its resemblance to a beaded bracelet), which we now define.

\begin{definition}
    We define a \emph{bracelet graph} $B_{r,n}$ via the following construction. Starting with a cycle $C_{2n+1}$ with vertices $v_0,\dots,v_{2n}$ in order, we blow up each of the vertices $v_{2i-1}$, where $1\le i\le n$, into a copy of $K_{r-1}$. Concretely, $B_{r,n}$ consists of an alternating sequence of $n+1$ single vertices (the ``little beads'') and $n$ $(r-1)$-cliques (the ``big beads''), such that each of the $(r-1)$-cliques is complete to the single vertices on both sides, and the two vertices on the ends are adjacent to each other.
\end{definition}

We label the vertices of $B_{r,n}$ by $w_0,\dots,w_{rn}$, where for each integer $i\in [0,n]$, $w_{ir}$ corresponds to $v_{2i}$, while for $i\in [n]$, $w_{(i-1)r+1},\dots,w_{ir-1}$ form the copy of $K_{r-1}$ that is complete to $w_{(i-1)r}$ and $w_{ir}$. See \cref{fig:Brn} for an example.

\begin{figure}[!h]
    \centering
    \begin{tikzpicture}[scale=1.5]
  
  \coordinate (1) at (0, 0);
  
  \coordinate (2a) at (1, 0.6);
  \coordinate (2b) at (1, -0.6);
  
  \coordinate (3) at (2, 0);
  
  \coordinate (4a) at (3, 0.6);
  \coordinate (4b) at (3, -0.6);
  
  \coordinate (5) at (4, 0);
  
  \coordinate (6a) at (5, 0.6);
  \coordinate (6b) at (5, -0.6);
  
  \coordinate (7) at (6, 0);
  
  \draw (1) -- (2a);
  \draw (1) -- (2b);
  \draw (2a) -- (2b);
  
  \draw (2a) -- (3);
  \draw (2b) -- (3);
  
  \draw (3) -- (4a);
  \draw (3) -- (4b);
  \draw (4a) -- (4b);
  
  \draw (4a) -- (5);
  \draw (4b) -- (5);
  
  \draw (5) -- (6a);
  \draw (5) -- (6b);
  \draw (6a) -- (6b);
  
  \draw (6a) -- (7);
  \draw (6b) -- (7);
  
  \draw (1) to[bend right=80] (7);
  
  \foreach \v in {1, 2a, 2b, 3, 4a, 4b, 5, 6a, 6b, 7} {
    \node[circle, fill=black, draw=black, line width=1.5pt, inner sep=2pt] at (\v) {};
  }
  
  \node[above=5pt] at (1) {$w_0$};
  \node[above=5pt] at (2a) {$w_1$};
  \node[below=5pt] at (2b) {$w_2$};
  \node[above=5pt] at (3) {$w_3$};
  \node[above=5pt] at (4a) {$w_4$};
  \node[below=5pt] at (4b) {$w_5$};
  \node[above=5pt] at (5) {$w_6$};
  \node[above=5pt] at (6a) {$w_7$};
  \node[below=5pt] at (6b) {$w_8$};
  \node[above=5pt] at (7) {$w_9$};
  
\end{tikzpicture}
    \caption{$B_{3,3}$}
    \label{fig:Brn}
\end{figure}

It is easy to see that $\chi(B_{r,n})=r+1$; indeed, in a proper $r$-coloring, each of the single vertices $w_0,w_r,\dots,w_{rn}$ would need to have the same color, which is impossible since $w_0$ and $w_{rn}$ are adjacent. When $r=2$, we simply recover $B_{r,n}=C_{2n+1}$.

Loosely following the notation of \cite{DayJohnson2017MulticolourRamseyOddCycles}, given a graph $G$ and a vertex $O\in V(G)$, we say that $G$ admits an $(r;n_1,\dots,n_k)$-rooted bracelet decomposition (written $(r;n_1,\dots,n_k)$-RBD for shorthand) with root $O$ if the edges of $G$ can be colored with $k$ colors such that for $1\le i\le k$, there is an edge-preserving map $\phi_i$ from the $i$th color class to $B_{r,n_i}$, where $O\in V(G)$ is the unique vertex mapped to $w_0$ in $B_{r,n_i}$. In other words, for each $1\le i\le k$, the $i$th color class is isomorphic to a subgraph of a blowup $B'$ of $B_{r,n_i}$, and $O\in V(G)$ is mapped by the isomorphism to the unique vertex in $B'$ corresponding to $w_0\in B_{r,n_i}$.

\begin{lemma}\label{lem:induct}
    If $K_{r^k+1}$ admits an $(r;n_1,\dots,n_k)$-RBD, then $K_{r^{k+1}+1}$ admits an $(r;n_1+1,n_2,\dots,n_k,2n_1)$-RBD.
\end{lemma}

\begin{proof}
    Given an $k$-edge-coloring $c_k$ of $K_{r^k+1}$ that exhibits the $(r;n_1,\dots,n_k)$-RBD, we construct an $(k+1)$-edge-coloring $c_k'$ of $K_{r^{k+1}+1}$ as follows. Let $V(K_{r^k+1})=\{O\}\cup U$, where $O$ is the root of the relevant RBD and $U$ consists of the other $r^k$ vertices. We identify $V(K_{r^{k+1}+1})$ with $\{O\}\cup U_1\cup \cdots \cup U_r$, where each $U_i=\{x^{(i)}:\: x\in U\}$ is a copy of $U$. For all $x,y\in U$ with $x\neq y$, let $c_k'(x^{(i)},y^{(j)})=c_k(x,y)$ for all $i,j\in [r]$. Likewise, let $c_k'(O,x^{(i)})=c_k(O,x)$ for $x\in U$. Finally, let $c_k'(x,x')$ be the new $(k+1)$-st color for all $x\in U$. 

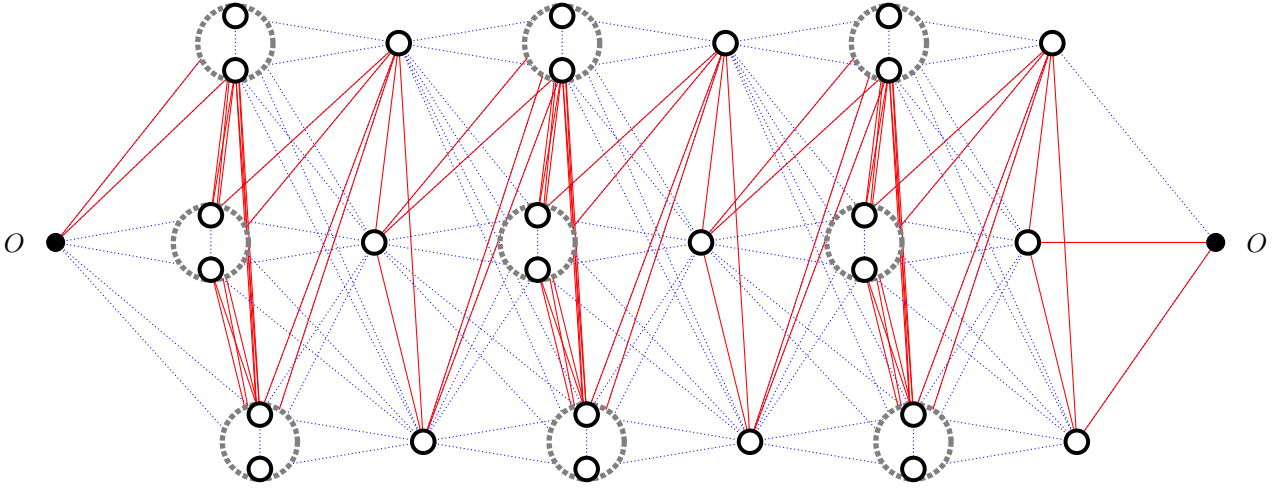
\begin{figure}
    \centering
 \begin{tikzpicture}[scale=1.2,xscale=1.8]
  \def\yshift{2.2}
  \def\xshift{-0.15}
  
  \coordinate (O-left) at (-0.1, 0);
  
  \coordinate (2a-1) at (1, 0.3 + \yshift);
  \coordinate (2b-1) at (1, -0.3 + \yshift);
  \coordinate (2c-1) at (1, 0 + \yshift);
  \coordinate (3-1) at (2, 0 + \yshift);
  \coordinate (4a-1) at (3, 0.3 + \yshift);
  \coordinate (4b-1) at (3, -0.3 + \yshift);
  \coordinate (4c-1) at (3, 0 + \yshift);
  \coordinate (5-1) at (4, 0 + \yshift);
  \coordinate (6a-1) at (5, 0.3 + \yshift);
  \coordinate (6b-1) at (5, -0.3 + \yshift);
  \coordinate (6c-1) at (5, 0 + \yshift);
  \coordinate (7-1) at (6, 0 + \yshift);
  
  \coordinate (2a-2) at (1 + \xshift, 0.3);
  \coordinate (2b-2) at (1 + \xshift, -0.3);
  \coordinate (2c-2) at (1 + \xshift, 0);
  \coordinate (3-2) at (2 + \xshift, 0);
  \coordinate (4a-2) at (3 + \xshift, 0.3);
  \coordinate (4b-2) at (3 + \xshift, -0.3);
  \coordinate (4c-2) at (3 + \xshift, 0);
  \coordinate (5-2) at (4 + \xshift, 0);
  \coordinate (6a-2) at (5 + \xshift, 0.3);
  \coordinate (6b-2) at (5 + \xshift, -0.3);
  \coordinate (6c-2) at (5 + \xshift, 0);
  \coordinate (7-2) at (6 + \xshift, 0);
  
  \coordinate (2a-3) at (1 - \xshift, 0.3 - \yshift);
  \coordinate (2b-3) at (1 - \xshift, -0.3 - \yshift);
  \coordinate (2c-3) at (1 - \xshift, 0 - \yshift);
  \coordinate (3-3) at (2 - \xshift, 0 - \yshift);
  \coordinate (4a-3) at (3 - \xshift, 0.3 - \yshift);
  \coordinate (4b-3) at (3 - \xshift, -0.3 - \yshift);
  \coordinate (4c-3) at (3 - \xshift, 0 - \yshift);
  \coordinate (5-3) at (4 - \xshift, 0 - \yshift);
  \coordinate (6a-3) at (5 - \xshift, 0.3 - \yshift);
  \coordinate (6b-3) at (5 - \xshift, -0.3 - \yshift);
  \coordinate (6c-3) at (5 - \xshift, 0 - \yshift);
  \coordinate (7-3) at (6 - \xshift, 0 - \yshift);
  
  \coordinate (O-right) at (7, 0);
  
  
  \foreach \i in {1,2,3} {
    \draw[blue, densely dotted] (O-left) -- (2a-\i);
    \draw[blue, densely dotted] (O-left) -- (2b-\i);
    \draw[blue, densely dotted] (2a-\i) -- (2b-\i);
    
    \draw[blue, densely dotted] (2a-\i) -- (3-\i);
    \draw[blue, densely dotted] (2b-\i) -- (3-\i);
    
    \draw[blue, densely dotted] (3-\i) -- (4a-\i);
    \draw[blue, densely dotted] (3-\i) -- (4b-\i);
    \draw[blue, densely dotted] (4a-\i) -- (4b-\i);
    
    \draw[blue, densely dotted] (4a-\i) -- (5-\i);
    \draw[blue, densely dotted] (4b-\i) -- (5-\i);
    
    \draw[blue, densely dotted] (5-\i) -- (6a-\i);
    \draw[blue, densely dotted] (5-\i) -- (6b-\i);
    \draw[blue, densely dotted] (6a-\i) -- (6b-\i);
    
    \draw[blue, densely dotted] (6a-\i) -- (7-\i);
    \draw[blue, densely dotted] (6b-\i) -- (7-\i);
    
    \draw[blue, densely dotted] (7-\i) -- (O-right);
  }
  
  \foreach \v in {2a, 2b, 3, 4a, 4b, 5, 6a, 6b, 7} {
    \draw[red] (\v-1) -- (\v-2);
    \draw[red] (\v-2) -- (\v-3);
    \draw[red] (\v-1) -- (\v-3);
  }
  
    \foreach \i in {1,2,3} {
      \foreach \j in {1,2,3} {
        \ifnum\i<\j
        \draw[red] (2a-\i) -- (2b-\j);
        \draw[red] (2b-\i) -- (2a-\j);
        \draw[red] (4a-\i) -- (4b-\j);
        \draw[red] (4b-\i) -- (4a-\j);
        \draw[red] (6a-\i) -- (6b-\j);
        \draw[red] (6b-\i) -- (6a-\j);
        \fi
      }
    }
  
  \foreach \i in {1,2,3} {
    \foreach \j in {1,2,3} {
      \ifnum\i<\j
        \draw[blue, densely dotted] (2a-\i) -- (3-\j);
        \draw[blue, densely dotted] (2b-\i) -- (3-\j);
        \draw[blue, densely dotted] (2a-\j) -- (3-\i);
        \draw[blue, densely dotted] (2b-\j) -- (3-\i);
        
        \draw[blue, densely dotted] (3-\i) -- (4a-\j);
        \draw[blue, densely dotted] (3-\i) -- (4b-\j);
        \draw[blue, densely dotted] (3-\j) -- (4a-\i);
        \draw[blue, densely dotted] (3-\j) -- (4b-\i);
        
        \draw[blue, densely dotted] (4a-\i) -- (5-\j);
        \draw[blue, densely dotted] (4b-\i) -- (5-\j);
        \draw[blue, densely dotted] (4a-\j) -- (5-\i);
        \draw[blue, densely dotted] (4b-\j) -- (5-\i);
        
        \draw[blue, densely dotted] (5-\i) -- (6a-\j);
        \draw[blue, densely dotted] (5-\i) -- (6b-\j);
        \draw[blue, densely dotted] (5-\j) -- (6a-\i);
        \draw[blue, densely dotted] (5-\j) -- (6b-\i);
        
        \draw[blue, densely dotted] (6a-\i) -- (7-\j);
        \draw[blue, densely dotted] (6b-\i) -- (7-\j);
        \draw[blue, densely dotted] (6a-\j) -- (7-\i);
        \draw[blue, densely dotted] (6b-\j) -- (7-\i);
      \fi
    }
  }

  \foreach \i in {2,3} {
    \foreach \j in {1} {
        \draw[red] (2a-\i) -- (3-\j);
        \draw[red] (2b-\i) -- (3-\j);
        
        \draw[red] (3-\i) -- (4a-\j);
        \draw[red] (3-\i) -- (4b-\j);
        
        \draw[red] (4a-\i) -- (5-\j);
        \draw[red] (4b-\i) -- (5-\j);
        
        \draw[red] (5-\i) -- (6a-\j);
        \draw[red] (5-\i) -- (6b-\j);
        
        \draw[red] (6a-\i) -- (7-\j);
        \draw[red] (6b-\i) -- (7-\j);
     
    }
  }
  \draw[red] (O-left) -- (2a-1);
  \draw[red] (O-left) -- (2b-1);
  \draw[red] (O-right) -- (7-2);
  \draw[red] (O-right) -- (7-3);

  
  
  \node[circle, fill=black, draw=black, line width=1.5pt, inner sep=2pt] at (O-left) {};
  \node[left=8pt] at (O-left) {$O$};
  
  \node[circle, fill=black, draw=black, line width=1.5pt, inner sep=2pt] at (O-right) {};
  \node[right=8pt] at (O-right) {$O$};

  \foreach \i in {1,2,3} {
    \foreach \v in {2c,4c,6c} {
        \node[circle, densely dotted, fill=white, draw=gray, line width=2pt, inner sep=10pt] at (\v-\i) {};
    }
    }
    
    \foreach \i in {1,2,3} {
        \draw[blue, densely dotted] (2a-\i) -- (2b-\i);
        \draw[blue, densely dotted] (4a-\i) -- (4b-\i);
        \draw[blue, densely dotted] (6a-\i) -- (6b-\i);
    
    }
  
  \foreach \i in {1,2,3} {
    \foreach \v/\label in {2a/$2a$, 2b/$2b$, 3/$3$, 4a/$4a$, 4b/$4b$, 5/$5$, 6a/$6a$, 6b/$6b$, 7/$7$} {
      \node[circle, fill=white, draw=black, line width=1.5pt, inner sep=3pt] at (\v-\i) {};
          }
  }
  
    \end{tikzpicture}
    \caption{Extending an $(r;n_1,\dots,n_k)$-RBD to an $(r;n_1+1,n_2,\dots,n_k,2n_1)$-RBD, for $r=3$, $n_1=3$. Here the blue dotted lines represent a color whose color class goes from a subgraph of a blowup of $B_{r,n_1}$ in $c_k$ to a subgraph of a blowup of $B_{r,n_1+1}$ in the new coloring $c_{k+1}$, while the red solid lines represent the new $(k+1)$st color in $c_{k+1}$, whose color class is a subgraph of a blowup of $B_{r,2n_1}$.}
        \label{fig:blowup}
\end{figure}

    In this coloring $c_k'$, for all $i\in [k]$, it remains true that the color class $G_i$ of color $i$ is a subgraph of a blowup of $B_{r,n_i}$ rooted at $O$. The new color, meanwhile, is $r$-partite. We now modify $c_k'$ into a new coloring $c_{k+1}$ that yields the desired $(r;n_1+1,n_2,\dots,n_k,2n_1)$-RBD, as follows. For convenience, call the first color blue, and call the new $(k+1)$-st color red. Fix a vertex partition that exhibits the fact that the blue graph $G_1$ is a (subgraph of a) rooted blowup of $B_{r,n_1}$, i.e. fix an edge-preserving map $\phi: G_1 \to B_{r,n_1}$ that sends $O$, and no other vertex, to $w_0$; the parts of the partition are the preimages of each vertex of $B_{r,n_1}$. Now for $x,y\in U$ with $x\neq y$ such that the edge $xy$ is blue in $c_k$, and for $i,j\in [r]$, recolor the edge $x^{(i)}y^{(j)}$ from blue to red in $c_{k+1}$ if and only if one of the following is true:
    \begin{itemize}
        \item $i\neq j$, and for some $m\in [n_1]$ and $a,b\in [r-1]$, we have $\phi(x)=w_{(m-1)r+a}, \phi(y)=w_{(m-1)r+b}$ (i.e. $\phi(x),\phi(y)$ are part of the same big bead $K_{r-1}$ in $B_{r,n_1}$).
        \item $i=1$, $j>1$, and for some $m\in [n_1]$ and $a\in [r-1]$, we have $\phi(x)=w_{mr}$ and $\phi(y)=w_{(m-1)r+a}$ (i.e. $\phi(x)$ is the little bead that comes after the big bead in $B_{r,n_1}$ that $\phi(y)$ is part of).
        \item $i=1$, $j>1$, and for some $m\in [n_1]$ and $a\in [r-1]$, we have $\phi(x)=w_{(m-1)r+a}$ and $\phi(y)=w_{(m-1)r}$ (i.e. $\phi(y)$ is the little bead that comes before the big bead in $B_{r,n_1}$ that $\phi(x)$ is part of).
    \end{itemize}
    We also recolor the edges between $O$ and $x^{(i)}$ red when either $i=1$ and $\phi(x)\in \{w_1,\dots,w_{r-1}\}$, or $i>1$ and $\phi(x)=w_{rn_1}$. See \cref{fig:blowup}.

    Then by construction, the red graph is a subgraph of a blowup of $B_{r,2n_1}$, where the (big and little) beads in the blue color class of the first copy $U_1$ of $U$, as a (subgraph of a) blowup of $B_{r,n_1}$, form the little beads in the red graph (with $O$ serving as the last little bead), while the corresponding beads in the blue color classes of the other copies of $U$ combine together into the big beads in the red graph. Meanwhile, in the blue graph, we can merge all but the first copy of $U$ back together, and identify each vertex $w_{i}$ in $\phi(U_1)$ with the vertex $w_{i+r}$ in $\phi(U_j)$ for $2\le j\le r$, to see that the blue graph is a subgraph of a blowup of $B_{r,n_1+1}$. This gives the desired $(r;n_1+1,n_2,\dots,n_k,2n_1)$-RBD of $K_{r^{k+1}+1}$.
\end{proof}
\begin{proof}[Proof of \cref{thm:lb}]
We inductively apply \cref{lem:induct}. Note that for $k=1$, $K_{r+1}$ trivially admits an $(r;1)$-RBD, since $B_{r,1}=K_{r+1}$. At each step $k\ge 1$, we sort the sequence $n_1\le \cdots\le n_k$ in increasing order and apply \cref{lem:induct} to extend the ``shortest'' color, which corresponds to $n_1$. 

When we iteratively apply the transformation given by \cref{lem:induct}, the sequence $(2n_1+1,\dots,2n_k+1)$ undergoes the same recursive transformation that the sequence $(r_1,\dots,r_k)$ defined in \cite{DayJohnson2017MulticolourRamseyOddCycles} undergoes in \cite[Lemma~5]{DayJohnson2017MulticolourRamseyOddCycles}. Thus, as reasoned in \cite[Corollary~6]{DayJohnson2017MulticolourRamseyOddCycles}, in $2^{O(t^2)}$ steps we can reach a mininum value $n_1$ of $>2^t$, meaning that in $k$ steps we can reach a minimum diameter of $2^{\Omega(\sqrt{\log k})}$ as desired.

\end{proof}

\section{Concluding remarks}
\label{sec:conclusion}
The upper and lower bounds on $L_r(k)$ for $r>2$, as in the case $r=2$, remain very far apart. It seems very plausible that the lower bound can be vastly improved by a better construction -- perhaps more easily for larger values of $r$ than for $r=2$, since the iterative construction given in \cref{sec:lb} does not make heavy use of the fact that we have more than two copies of the construction from the previous step to work with. The upper bound, meanwhile, is entangled by the relationship between the Lov\'asz theta function and bounds on Shannon capacities, just as it is in the $r=2$ case (see the discussion in \cite[Section~3]{JanzerYip2026ShortMonochromaticOddCycles}). Some very recent developments on the latter front (see e.g. \cite{itty2026improved}) may thus hold promise for further progress on these bounds.

In the case $r=2$, the authors of \cite{ACJMR2026} showed that the upper bound on the diameter of the smallest monochromatic odd cycle in a $k$-edge-coloring of $K_n$ can be significantly improved using different methods in the general setting where $n=(1+\delta)r^k$, as long as $\delta$ does not go to zero too quickly. Their arguments, specifically \cite[Lemma~2.1]{ACJMR2026}, can be used to analogously conclude that the bound from \cref{thm:ub-gen} is far from tight for relatively large $\delta$ in the case $r>2$.

\section*{Acknowledgements}
The first author would like to thank Noga Alon and Oliver Janzer for separately bringing this interesting question to his attention. The authors would also like to thank Ting-Wei Chao and Lisa Sauermann for the helpful discussions and insights they offered for the problem. ChatGPT 5.5 was used as a sounding board for developing the proof of \cref{thm:ub}, and Claude Haiku 4.5 was used to assist in creating the tikz diagrams in \cref{fig:Brn} and \cref{fig:blowup}.

\printbibliography

\end{document}